\documentclass{interact}

\usepackage{epstopdf}
\usepackage[caption=false]{subfig}

\usepackage[colorlinks,citecolor=blue,urlcolor=blue]{hyperref}

\usepackage[numbers,sort&compress]{natbib}
\bibpunct[, ]{[}{]}{,}{n}{,}{,}
\renewcommand\bibfont{\fontsize{10}{12}\selectfont}
\makeatletter
\def\NAT@def@citea{\def\@citea{\NAT@separator}}
\makeatother

\theoremstyle{plain}
\newtheorem{theorem}{Theorem}
\newtheorem{lemma}{Lemma}
\theoremstyle{remark}
\newtheorem{rem}{Remark}
\newtheorem{definition}{Definition}
\newtheorem{assumption}{Assumption}

\begin{document}


\title{On LSE in regression model for long-range dependent random fields on spheres.}

\author{
\name{Vo Anh\textsuperscript{a,b}, Andriy Olenko\textsuperscript{c}\thanks{CONTACT Andriy Olenko. Email: a.olenko@latrobe.edu.au} and Volodymyr Vaskovych\textsuperscript{c}}
\affil{\textsuperscript{a}School of Mathematical Sciences, Queensland University of Technology,	Brisbane, Queensland, 4001, Australia\\ \textsuperscript{b}Faculty of Science, Engineering and Technology, Swinburne University of Technology, Victoria, 3122, Australia\\ \textsuperscript{c}Department of Mathematics and Statistics,
	La Trobe University, Melbourne, 3086, Australia}
}

\maketitle

\begin{abstract}
We study the asymptotic behaviour of least squares estimators in regression models for long-range dependent random fields observed on spheres. The least squares estimator can be given as a weighted functional of long-range dependent random fields. It is known that in this scenario the limits can be non-Gaussian. We derive the limit distribution and the corresponding rate of convergence for the estimators. The results were obtained under rather general assumptions on the random fields. Simulation studies were conducted to support theoretical findings.
\end{abstract}

\begin{amscode}
	62H11, 62J05, 60G60.
\end{amscode}

\section{Introduction} 
Studying the asymptotic behavior of the regression parameter estimators is an important topic in Spatial Statistics, see \cite{RM}, \cite{LRM}, \cite{Miu}. A potential area of application of the results in this paper is cosmology, and more specifically, models involving the cosmic microwave background (CMB). The recent data obtained by the Wilkinson Microwave Anisotropy Probe and Planck missions, combined with data that will be obtained during the future missions, will allow to predict the behaviour of CMB at increasing distances that are far beyond the reach of any modern telescopes. Another area of application is studying the porosity of mineral resources, see, for example, \cite{doyen} about data and analysis of the porosity of hydrocarbon reservoirs. Some other potential applications of the model discussed in this article might involve 3D rupture models, soil fertility and paleomagnetism.

One of the first works that developed statistical theory of random fields was the book by \cite{Yadr}. Direct probability techniques for models involving random fields were proposed and various estimators of regression parameters were investigated. The asymptotic behaviour of some of these estimators was studied in \cite{IvLeo}. Both Gaussian and non-Gaussian limits were obtained for particular models. No results about rates of convergence were given.

In this paper we consider the regression model $\xi(x) = ah(x) + \psi(x)$, where $h(\cdot)$ is a known deterministic function, $a$ is an unknown parameter and the error $\psi(\cdot)$ is a random field. For results concerning such regression models see \cite{cohen} and references therein. Since there are numerous situations in practice when random processes and fields are non-Gaussian, we consider a particular case when $\psi(\cdot)$ is a non-linear function of a Gaussian random field. Such random fields are very common in non-Gaussian modeling since they can be analyzed using Wiener chaos expansions and in many cases offer a good data approximation, see \cite{Oli, Vio}. In this article, we consider the underlying Gaussian random fields to be long-range dependent.

Long-range dependence is a well-established phenomenon that can be observed in various fields such as hydrology, agriculture, image analysis, earth sciences, cosmology, just to name a few. For this reason, models involving  long-range dependent random fields have been an object of statistical interest for years, see \cite{frias, Douk, Iv, Wack}. 

Various earth science and cosmology applications require studying random fields defined on surfaces, for example, see \cite{Guin, leorm}. In this paper we consider random fields that are defined on expanding spheres. One can find detailed information about such spherical random fields in \cite{Mar, Yadr}.

The regression model studied in this paper was first considered in \cite{Yadr}. For this model the best linear unbiased estimator (BLUE) of the unknown parameter $a$ was derived. In \cite{YadrIl} the least squares estimator (LSE) of $a$ was obtained. Its mean-square efficiency was compared to BLUE and it was shown that, apart from some degenerate cases, LSE is less efficient than BLUE. However, for most functions $h(\cdot)$ BLUE is much more difficult to compute than LSE. For this reason using LSE is preferable to BLUE in practice. No results concerning limit distributions of the regression parameter estimators were obtained in \cite{YadrIl, Yadr} or other literature known to the authors. This article investigates the limit distribution of LSE and its rate of convergence.

For the particular case $h(\cdot) \equiv 1$, some approaches to study non-linear functionals of random fields on surfaces and rates of convergence were proposed in \cite{surfNew}. This article shows how the methodology developed in \cite{surfNew} can be extended to the general case of weighted non-linear functionals of random fields, i.e. the case of arbitrary functions $h(\cdot)$.  To reduce repetitions to a minimum, in this article we present only those parts of proofs that are new or require modifications of results in \cite{surfNew}.

The main goal of this article is to demonstrate potential applications of the developed methodology to statistical problems. We focus on the LSE for a regression model, however there are numerous other statistical problems that can be studied using weighted non-linear functionals of random fields. For example, by choosing the appropriate integrand functions, various characteristics of random fields, such as moments, can be represented by the considered functionals. In \cite{surfNew} it was shown that Minkowski functionals can be obtained by choosing indicator functions as the integrands. Another important statistical application is assigning different weights to different observation regions. In spatial statistics it is often used to control the edge effects. Also, assigning weights to observations or observation regions is an  important case of data tapering that is often used to reduce effects of less reliable observations. 

Section~\ref{sec4} presents detailed simulation studies that back up the theory and suggest new research questions.  To the best of our knowledge, simulation studies  have not been done in the available literature on convergence rates in non-central limit theorems for long-range dependent random fields. The numerical study suggests that the exponential rate of convergence obtained in all known theoretical studies might be improved. For the cases of a sphere and a cube, various observation windows and weight functions were investigated. The detailed methodology for practical simulations and studying weighted functionals of random fields observed on surfaces is provided. The corresponding R codes for one-core and parallel computing are freely available on-line and can be used by researchers in this field.  

To obtain the main results some fine properties of the Fourier transforms on surfaces are employed. Note that the Fourier transforms were more frequently studied on solid bodies than on surfaces. The assumptions for obtaining the required rate of decay are weaker in the case of solid bodies. Therefore, all results derived in this article can be analogously obtained for solid bodies.

As mentioned earlier, to study the asymptotic behaviour of LSE we investigate weighted non-linear functionals of Gaussian random fields. Non-linear functionals of random processes were considered in \cite{IvLeo, LeoSi}. In \cite{IvLeo} only the case of short-range dependent random processes and Gaussian asymptotics for weighted functionals was studied. In \cite{LeoSi} non-Gaussian asymptotics were obtained for long-range dependent random processes. The results in \cite{LeoSi} can be obtained from the methodology of this article as random processes can be considered as random fields indexed by a one-dimensional Euclidean space.
For random fields and the simplest particular case $h(\cdot) \equiv 1$  it was shown that the functionals can produce non-Gaussian limits, see  \cite{Dob,Taq1,Taq2}. The more detailed overview of this particular case was given in \cite{New}. In the general case, weighted linear functionals of the Gaussian random fields were studied in \cite{Ole}.  The assumptions on weighted functions used in this article are weaker than those used in \cite{Ole}. Moreover, we do not apply any conditions on Fourier transforms of weight functions. The known Gaussian asymptotics are a particular case when integrand functions are linear.  None of the above studies provided the rate of convergence to the obtained limits. To the best of our knowledge, this paper is  the first work that provides the rate of convergence for weighted functionals and LSE of long-range dependent random fields on hypersurfaces. 

The article is organized as follows. In Section~\ref{sec1} we recall some basic definitions of the spectral theory of random fields. Section~\ref{sec2} presents the model and states the results. Proofs of the results are provided in Section~\ref{sec3}. In Section~\ref{sec4} some simulation studies are presented to confirm theoretical findings. Conclusions and directions for future research are stated in Section~\ref{sec5}.

\section{Definitions}\label{sec1}
In this section we provide the main definitions that are used in this article.

In what follows $|\cdot|$ and $\|\cdot\|$ denote the Lebesgue measure and the Euclidean distance in $\mathbb{R}^d$, $d\geq 2$, respectively. Let $B(y, s)$ be a $d$-dimensional ball with centre $y$ and radius $s$, and let $S(r)$ be a $(d-1)$-dimensional sphere in $\mathbb{R}^{d}$ with the centre at the origin and radius $r$.  We use the symbols $C$ and $\delta$ to denote constants which are not important for our exposition. Moreover, the same symbol may be used for different constants appearing in the same proof.

Let $(\Omega ,\mathcal{F},P)$ be a complete probability space and let $T \subseteq \mathbb{R}^d$ be a set.
\begin{definition}\cite{Iv}
	A random field is a function $\eta(w, x): \Omega\times T \rightarrow \mathbb{R}$ such that $\eta(w, x)$ is a random variable for any $x\in T$. It will also be denoted as  $\eta (x),\ x\in T.$
\end{definition}

\begin{definition}\cite{Iv}
	A random field satisfying $\mathbf{E}\eta^2(x)<\infty$ is called homogeneous in the wide sense if its mathematical expectation $m(x) = \mathbf{E}\eta(x)$ and
	covariance function ${\rm{B}}(x, y) = {\mathbf{E}}\left[\eta(x)-m(x)\right]\left[\eta(y)-m(y)\right]$ are invariant with respect to the group of shifts $\tau$ in $\mathbb{R}^d$, that is, $m(x) = m(x+\tau)$, ${\rm{B}}(x, y) = {\rm{B}}(x+\tau, y+\tau)$ for any $x,y \in T$.
	 
	It means that $\mathbf{E}\eta(x) = c = const$, and the covariance
	 ${\rm{B}}(x, y)$ depends only on the difference $x - y$.
\end{definition}

\begin{definition}\cite{Iv}
	A random field satisfying $\mathbf{E}\eta^2(x)<\infty$ is called isotropic in the wide sense if its mathematical expectation $m(x)$ and covariance
		function ${\rm{B}}(x, y)$ are invariant with respect to rotations.
\end{definition}

In this article homogeneity and isotropy in the wide sense will be simply called homogeneity and isotropy. 

Let us consider a measurable mean-square continuous zero-mean homogeneous
isotropic real-valued random field $\eta (x),\ x\in \mathbb{R}^{d},$ defined
on a probability space $(\Omega ,\mathcal{F},P),$ with the covariance function 
\begin{equation*}
\text{\rm{B}}(r):=\mathrm{Cov}\left( \eta (x),\eta (y)\right)
=\int_{0}^{\infty }Y_{d}(rz)\,\mathrm{d}\Phi (z),\ x,y\in \mathbb{R}^{d},
\end{equation*}%
where $r:=\left\Vert x-y\right\Vert ,$ $\Phi (\cdot )$ is the isotropic
spectral measure, the function $Y_{d}(\cdot )$ is defined by 
\begin{equation*}
Y_{d}(z):=2^{(d-2)/2}\Gamma \left( \frac{d}{2}\right) \ J_{(d-2)/2}(z)\
z^{(2-d)/2},\quad z\geq 0,
\end{equation*}%
and $J_{(d-2)/2}(\cdot )$ is the Bessel function of the first kind of order $%
(d-2)/2.$

\begin{definition}
	The random field $\eta (x),$ $x\in \mathbb{R}^{d},$ defined above is said
	to possess an absolutely continuous spectrum if there exists a function $%
	f(\cdot )$ such that 
	\begin{equation*}
	\Phi (z)=2\pi ^{d/2}\Gamma ^{-1}\left( d/2\right) \int_{0}^{z}u^{d-1}f(u)\,%
	\mathrm{d}u,\quad z\geq 0,\quad u^{d-1}f(u)\in L_{1}(\mathbb{R}_{+}).
	\end{equation*}%
	The function $f(\cdot )$ is called the isotropic spectral density function
	of the field~$\eta (x).$ The field~$\eta (x)$ with an absolutely continuous spectrum has the following representation 
	\begin{equation*}
	\eta (x)=\int_{\mathbb{R}^d}e^{i<\lambda ,x>}\sqrt{f\left( \left\| \lambda
		\right\| \right) }W(\mathrm{d}\lambda ),
	\end{equation*}
	where $W(\cdot )$ is the complex Gaussian white noise random measure on $%
	\mathbb{R}^d.$
\end{definition}

Let $H_{k}(u)$, $k\geq 0$, $u\in \mathbb{R}$, be the Hermite polynomials,
see \cite{pec}. These polynomials form a complete orthogonal system in the Hilbert space 
\begin{equation*}
{L}_{2}(\mathbb{R},\phi (w)\,dw)=\left\{ G:\int_{\mathbb{R}}G^{2}(w)\phi
(w)\,\mathrm{d}w<\infty \right\} ,\quad \phi (w):=\frac{1}{\sqrt{2\pi }}e^{-%
	\frac{w^{2}}{2}}.
\end{equation*}

An arbitrary function $G\in {L}_{2}(\mathbb{R},\phi(w )\, dw)$ admits the
mean-square convergent expansion 
\begin{equation*}  \label{herm}
G(w)=\sum_{j=0}^{\infty }\frac{C_{j}H_{j}(w) }{j !}, \qquad C_{j }:=\int_{%
	\mathbb{R}}G(w)H_{j }(w)\phi ( w )\,\mathrm{d}w.
\end{equation*}

By Parseval's identity 
\begin{equation*}  \label{par}
\sum_{j=0}^\infty\frac{C_{j}^{2}}{j !} =\int_{\mathbb{R}}G^2(w) \phi ( w )\,%
\mathrm{d}w.
\end{equation*}

It will be shown that studying asymptotics of non-linear functionals defined by a function $G\in {L}_{2}(\mathbb{R},\phi(w )\, dw)$ can be done by investigating leading terms in their expansions. The following definitions and remarks provide basic notations and tools to formulate and prove these asymptotic results.

\begin{definition} \rm{\cite{Taq1}} Let $G\in {L}_{2}(\mathbb{R},\phi (w)\,dw)$ and
	assume there exists an integer $\kappa \in \mathbb{N}$ such that $C_{j}=0$, for all $%
	0\leq j\leq \kappa -1,$ but $C_{\kappa }\neq 0.$ Then $\kappa $ is called
	the Hermite rank of $G(\cdot )$ and is denoted by $H\mbox{rank}\,G.$
\end{definition}

The following remark gives a well known property of Hermite polynomials of Gaussian vectors, see~\cite{Iv}.

\begin{rem}{\label{rem1}}
	If $(\xi _{1},\ldots ,\xi _{2p})$ is a $2p$-dimensional
	zero-mean Gaussian vector with 
	\begin{equation*}
	\mathbf{E}\xi _{j}\xi _{k}=%
	\begin{cases}
	1, & \mbox{if }k=j, \\ 
	r_{j}, & \mbox{if }k=j+p\ \mbox{and }1\leq j\leq p, \\ 
	0, & \mbox{otherwise,}%
	\end{cases}%
	\end{equation*}%
	then 
	\begin{equation}\label{hvar}
	\mathbf{E}\ \prod_{j=1}^{p}H_{k_{j}}(\xi _{j})H_{m_{j}}(\xi
	_{j+p})=\prod_{j=1}^{p}\delta _{k_{j}}^{m_{j}}\ k_{j}!\ r_{j}^{k_{j}},
	\end{equation}
	where   $\delta _{k}^{m}=%
	\begin{cases}
		1, & \mbox{if }k=m, \\ 
		0, & \mbox{if }k\neq m,
	\end{cases}$ \ is the Dirac delta function.
\end{rem} 

Let $\Delta$ be a bounded set in $\mathbb{R}^d,\, d\geq 2,$ with a boundary $\partial\Delta$. Let $\Delta(r)$, $r > 0,$ be the homothetic image of the set $\Delta$ with the centre of homothety at the origin and the
coefficient $r > 0$, that is $|\Delta(r)| = r^d|\Delta|$. Let $\sigma(\cdot)$ be the $d - 1$-dimensional Lebesgue measure on  $\partial\Delta(r)$. Let $U$ and $V$ be two independent and uniformly distributed random vectors on the hypersurface $
\partial\Delta (r)$. We denote by $\psi _{\Delta (r)}(\rho
),$ $\rho \geq 0,$ the probability density function of the distance $\left\| U-V\right\| $ between $U$ and $V.$   Note that $\psi _{\Delta (r)}(\rho
)=0$ if $\rho > diam\left\{ \partial\Delta (r)
\right\}.$
Using the above notations, we obtain the  representation
\[
\int\limits_{\partial\Delta (r)}\int\limits_{\partial\Delta (r)}G(\left\| x-y\right\| )\sigma({\rm d}x)\,
\sigma({\rm d}y)=
\left| \partial\Delta \right| ^{2}r^{2d-2}\mathrm{E}\ G(\left\| U-V\right\| )=
\]
\begin{equation}\label{dint}
=\left| \partial\Delta \right| ^{2}r^{2d-2}\int_{0}^{diam\left\{ \partial\Delta
	(r)\right\} }G(\rho )\ \psi _{\Delta (r)}(\rho )d\rho.
\end{equation}

For various hypersurfaces $\partial\Delta$ explicit expressions for $\psi _{\Delta (r)}(\cdot)$ are presented in \cite{Iv}. In the case of spheres it takes the following form. 

\begin{rem}
	If $\partial\Delta(r) = S(r)$, then 	
	\[\psi _{\Delta (r)}(\rho) = \frac{1}{\sqrt{\pi}}\Gamma\left(\frac{d}{2}\right)\Gamma^{-1}\left(\frac{d-1}{2}\right)r^{1-d}\rho^{d-2}\left(1 - \frac{\rho ^2}{4u^2}\right)^{\frac{d-3}{2}}, \quad 0 < \rho < 2r.\]
\end{rem}

\begin{definition}
	\rm{\cite{bin}} A measurable function $L:(0,\infty )\rightarrow
	(0,\infty )$ is said to be slowly varying at infinity if for all $t>0$%
	\begin{equation*}
	\lim\limits_{\lambda \rightarrow \infty }\frac{L(\lambda t)}{L(\lambda )}=1.
	\end{equation*}
\end{definition}

If $L(\cdot )$ varies slowly, then $r^{a}L(r)\rightarrow \infty ,$ $%
r^{-a}L(r)\rightarrow 0$ for an arbitrary $a>0$ when $r\rightarrow \infty ,$
see Proposition~1.3.6 in \cite{bin}.

\begin{definition}{\cite{bin}}
	A measurable function $g:(0,\infty )\rightarrow (0,\infty )$ is
	said to be regularly varying at infinity, denoted $g(\cdot )\in R_{\tau }$,
	if there exists $\tau $ such that, for all $t>0,$ it holds that
	\begin{equation*}
	\lim\limits_{\lambda \rightarrow \infty }\frac{g(\lambda t)}{g(\lambda )}%
	=t^{\tau }.
	\end{equation*}
\end{definition}

\begin{definition}\label{sr2}{\cite{bin}}
	Let	$g:(0,\infty )\rightarrow (0,\infty )$ be a measurable function and $g(x)\rightarrow 0$
	as $x\rightarrow \infty$. A slowly varying function $L(\cdot )$ is said to be slowly varying with
	remainder of type 2, or that it belongs to the class SR2, if 
	\begin{equation*}
	\forall \lambda >1:\quad \frac{L(\lambda x)}{L(x)}-1\sim k(\lambda
	)g(x),\quad x\rightarrow \infty ,
	\end{equation*}%
	for some function $k(\cdot )$.
	
	If there exists $\lambda $ such that $k(\lambda )\neq 0$ and $k(\lambda \mu
	)\neq k(\mu )$ for all $\mu $, then $g(\cdot )\in R_{\tau }$ for some $\tau
	\leq 0$ and $k(\lambda )=ch_{\tau }(\lambda )$, where 
	\begin{equation}
	h_{\tau }(\lambda )=%
	\begin{cases}
	\ln (\lambda ),\quad if\;\tau =0, \\ 
	\frac{\lambda ^{\tau }-1}{\tau },\quad \,if\;\tau \neq 0.%
	\end{cases}
	\label{htau}
	\end{equation}
\end{definition}

\begin{rem}\label{log}
	An example of a function that satisfies Definition~\ref{sr2} for $\tau = 0$ is $L(x) = \ln(x).$ Indeed, 
	\[ 
	\frac{L(\lambda x)}{L(x)} - 1 = \frac{\ln(\lambda) + \ln(x)}{\ln(x)} -1 = \ln(\lambda)\cdot\frac{1}{\ln(x)}.
	\]
\end{rem}

\begin{definition}
	Let $Y_1$ and $Y_2$ be arbitrary random variables. The uniform (Kolmogorov)
	metric for the distributions of $Y_1$ and $Y_2$ is defined by 
	\begin{equation*}
	{\rho}\left( Y_1,Y_2\right) =\underset{z\in \mathbb{R}}{\sup }\left| P\left(
	Y_1\leq z\right) -P\left( Y_2\leq z\right) \right| .
	\end{equation*}
\end{definition}

The next result follows from Lemma~1.8 in \cite{pet}.

\begin{lemma}
	\label{lem1} If $X,Y$ and $Z$ are arbitrary random variables, then for any $%
	\varepsilon >0$ 
	\begin{equation*}
	\rho \left( X+Y,Z\right) \leq {\rho }(X,Z)+\rho \left( Z+\varepsilon
	,Z\right) +P\left( \left\vert Y\right\vert \geq \varepsilon \right) .
	\end{equation*}
\end{lemma}

\section{Model and results}\label{sec2}

Let us consider the random field 
\[\xi(x) = a_rh(x) + \psi(x), \qquad x\in S(r),\quad r > 0,\]
where $h(\cdot)$ is a known deterministic function, $a_r$ is an unknown parameter and $\psi(\cdot)$ is a homogeneous isotropic mean-square random field with mean 0.

Suppose that $h(x): = h_{rad}(\|x\|)h_{sp}\left(\frac{x}{\|x\|}\right),$ where $h_{rad}(\cdot)$ is a radial function such that $h_{rad}(r) \neq 0$ for all $r > 0$ and $h_{sp}(\cdot)\not\equiv 0$ is a smooth bounded function defined on the unit sphere $S(1)$.

In this case, the least squares estimator (LSE) of the coefficient $a_r$ has the explicit form, see \cite{Yadr},
\[
\hat{a}_r = \frac{\displaystyle\int\limits_{S(r)}\xi(x)h(x)\sigma({\rm d}x)}{\displaystyle\int\limits_{S(r)}h^2(x)\sigma({\rm d}x)} = \frac{\displaystyle\int\limits_{S(r)}\xi(x)h_{rad}(\|x\|)h_{sp}\left(\frac{x}{\|x\|}\right)\sigma({\rm d}x)}{\displaystyle\int\limits_{S(r)}h^2_{rad}(\|x\|)h_{sp}^2\left(\frac{x}{\|x\|}\right)\sigma({\rm d}x)}\]
\[ = \frac{\displaystyle\int\limits_{S(r)}\xi(x)h_{sp}\left(\frac{x}{\|x\|}\right)\sigma({\rm d}x)}{h_{rad}(r)\displaystyle\int\limits_{S(r)}h_{sp}^2\left(\frac{x}{\|x\|}\right)\sigma({\rm d}x)}.
\]

Let $\psi(x) = G(\eta(x)),$ where $G\in {L}_{2}(\mathbb{R},\phi(w )\, dw)$ and $\eta(x)$ is a random field that satisfies the following assumptions.

\begin{assumption}
	\label{ass1} Let $\eta (x),$ $x\in \mathbb{R}^{d}$, be a homogeneous
	isotropic Gaussian random field with $\mathbf{E}\eta (x)=0$ and a covariance
	function $B(x)$ such that
	\begin{equation*}
	B(0)=1,\quad B(x)=\mathbf{E}\eta (0) \eta (x)= \left\Vert x\right\Vert
	^{-\alpha }L_{0}(\left\Vert x\right\Vert ),
	\end{equation*}
	where $L_{0}(\left\Vert \cdot\right\Vert )$ is a function slowly varying at
	infinity.
\end{assumption}

This assumption is a classical way to introduce hyperbolically decaying dependencies between observations, see \cite{Dob, Douk, Taq2, Iv} and references therein. Random fields satisfying this assumption for $\alpha > d$ are weakly-dependent. If $\alpha\in(0, d)$ then the long-range dependence case is considered.

\begin{assumption}
	\label{ass2} The random field $\eta(x),$ $x \in \mathbb{R}^d,$ has the
	spectral density
	\[
	f(\left\| \lambda \right\| )= c_2(d,\alpha )\left\| \lambda \right\|
	^{\alpha -d}L\left( \frac 1{\left\| \lambda \right\| }\right),
	\]
	where $c_2(d,\alpha ):=\frac{\Gamma \left( \frac{d-\alpha }2\right) }{2^\alpha \pi
		^{d/2}\Gamma \left( \frac \alpha 2\right) },$ and $L(\left\Vert \cdot\right\Vert )$ is a locally bounded function slowly varying at infinity which satisfies for sufficiently large $r$ the
	condition 
	\begin{equation}  \label{gr}
	\left|1-\frac{L(tr)}{L(r)}\right|\le C\,g(r)h_{\tau}(t),\ t\ge 1,
	\end{equation}
	where $g(\cdot) \in R_{\tau} , \tau \le 0$, such that $g(x) \to 0, \ x \to
	\infty$, and $h_{\tau}(t)$ is defined by~\rm{(\ref{htau}).}
\end{assumption}

Long-range dependence is usually introduced by requiring the hyperbolic decay of covariance functions or power-type singularity of the corresponding spectral densities. For many real data these two definitions are operationally equivalent, see also Tauberian-Abelian theorems in \cite{leoole}. However, there are cases when Assumption~\ref{ass2} does not follow from Assumption~\ref{ass1}, see \cite{leoole}. To make all following results rigorous, we require that the both assumptions hold.

Examples of popular classes of random fields that satisfy Assumption~\ref{ass1} and~\ref{ass2} simultaneously are Bessel, Cauchy, and Linnik random fields. 

\begin{rem}\label{eql}
	By Tauberian and Abelian theorems, see \cite{leoole}, for $L_0(\cdot)$ and $L(\cdot)$ given in  Assumptions~\ref{ass1} and \ref{ass2} it holds  $L_0(r) \sim L(r),$ $r\to +\infty.$  
\end{rem}

In \cite{New} the following property of slowly varying functions satisfying conditions (\ref{gr}) was proven.

\begin{rem}
	~\label{rem0} If $L(\cdot)$ satisfies {\rm{(\ref{gr})}}, then for any $%
	k\in \mathbb{N}$, $\delta >0$, and sufficiently large~$r$ 
	\begin{equation*}
	\left\vert 1-\frac{L^{k/2}(tr)}{L^{k/2}(r)}\right\vert \leq C\,g(r)h_{\tau
	}(t)t^{\delta },\ t\geq 1.
	\end{equation*}
\end{rem}

Since $\psi(x) = G(\eta(x)),$ the LSE has the form 

\[
\hat{a}_r = a_r + \frac{\displaystyle\int\limits_{S(r)}G(\eta(x))h_{sp}\left(\frac{x}{\|x\|}\right)\sigma({\rm d}x)}{h_{rad}(r)\displaystyle\int\limits_{S(r)}h_{sp}^2\left(\frac{x}{\|x\|}\right)\sigma({\rm d}x)}.
\]

Let $H\mathrm{rank}\,G=\kappa$. Our main object of interest is the random variable 
\[
X_{r, G}: = c_h(r)c_r(d,\alpha)(\hat{a}_r - a_r),
\]
where 
\[c_h(r) := h^{-1}_{rad}(r)\int\limits_{S(r)}h^2(x)\sigma({\rm d}x) \text{ and } c_r(d,\alpha) = \frac{\kappa!c_2^{-\kappa/2}(d,\alpha )}{C_\kappa r^{d-1-\frac{\kappa\alpha}{2}}L^{%
		\frac{\kappa}{2}}(r)}.\]

It is straightforward to see that 
\[
X_{r, G} = c_r(d, \alpha)\int\limits_{S(r)}G(\eta(x))h_{sp}\left(\frac{x}{\|x\|}\right)\sigma({\rm d}x)\]
\[ = c_r(d, \alpha)\frac{C_\kappa}{\kappa!}
\int\limits_{S(r)}H_\kappa (\eta (x))h_{sp}\left(\frac{x}{\|x\|}\right)\,\sigma(\mathrm{d}x)
\]
\begin{equation}\label{intf}
+ c_r(d, \alpha)\sum_{j\geq \kappa + 1}\frac{C_j}{j!}
\int\limits_{S(r)}H_j (\eta (x))h_{sp}\left(\frac{x}{\|x\|}\right)\,\sigma(\mathrm{d}x) =: X_{r,\kappa}+V_r.
\end{equation}

\begin{theorem}\label{th1}
	Suppose that $H\mathrm{rank}\,G=\kappa \in \mathbb{N}$ and $\eta (x),$ $x\in \mathbb{R}^{d},$
	satisfies Assumption~{\rm\ref{ass1}} for $\alpha\in(0, (d-1)/\kappa)$. If at least one of the following random variables
	\begin{equation*}
	\frac{X_{r, G}}{\sqrt{\mathbf{Var}\text{ }X_{r, G}}},\quad \frac{X_{r, G}}{\sqrt{%
			\mathbf{Var}\ X_{r,\kappa }}}\quad \mbox{and}\quad \frac{X_{r,\kappa }}{%
		\sqrt{\mathbf{Var}\ X_{r,\kappa }}},
	\end{equation*}%
	has a limit distribution, then the limit distributions of the other random variables also exist and
	they coincide when $r\rightarrow \infty .$
\end{theorem}

By Theorem~\ref{th1} to study limit distributions one can use $X_{r,\kappa }$ instead of $X_{r, G}$. 

In Theorem~\ref{th2} below, we show that the limit distribution of $X_{r,\kappa }$ is a Hermite-type random variable, which can be represented by a Wiener-It\^{o} integral. Let
\begin{equation*}
K(x):=\int\limits_{S(1) }e^{i<x,u>} h_{sp}(u)\sigma(\mathrm{d}u),\quad
x\in\mathbb{R}^{d},
\end{equation*}
be the Fourier transform of the function $h_{sp}(\cdot)$ over the $(d-1)$-dimensional sphere with radius 1. Using the decay properties of the Fourier transform one can prove the following.

\begin{lemma}\label{finint}
	If $\tau_1,...,\tau_\kappa,$ $\kappa\ge 1,$ are such positive constants, that $\sum_{i=1}^\kappa \tau_i <d-1,$  then
	\begin{equation}\label{finv}
	\int\limits_{\mathbb{R}^{d\kappa}}|K(\lambda _1+\cdots
	+\lambda _\kappa)|^2 \frac{\mathrm{d}\lambda _1\ldots \mathrm{d}\lambda _\kappa}{\left\| \lambda
		_1\right\| ^{d-\tau_1}\cdots \left\| \lambda _\kappa\right\| ^{d-\tau_\kappa}}<\infty .
	\end{equation}
\end{lemma}

Let 
\[
X_\kappa :={\int\limits_{\mathbb{R}^{d\kappa}}}^{
	\prime}K(\lambda _1+\cdots +\lambda _\kappa)\frac{W(\mathrm{d}%
	\lambda _1)\ldots W(\mathrm{d}\lambda _\kappa)}{\left\| \lambda _1\right\|
	^{(d-\alpha )/2}\ldots \left\| \lambda _\kappa\right\| ^{(d-\alpha )/2}},
\]
where ${\int\limits_{\mathbb{R}^{d\kappa}}}^{\prime}$ denotes the multiple Wiener-It\^{o} integral.

\begin{theorem}\label{th2} Let $\eta(x),$ $x\in \mathbb{R}^d,$ be a homogeneous isotropic Gaussian random
	field with $\mathbf{E}\eta(x)=0.$  If Assumptions~{\rm{\ref{ass1}}} and {\rm{\ref{ass2}}} hold, $\alpha \in (0,(d-1)/\kappa),$ and $H\mathrm{rank}\,G=\kappa\in\mathbb{N},$  then for $r\to \infty$ the random variables $X_{r,\kappa }$ converge weakly to $X_\kappa$.
	
\end{theorem}

To obtain the rate of convergence of $X_{r, G}$ to $X_\kappa$ we will use some fine properties of Hermite-type distributions. We will denote the Wiener-It\^{o} integrals of rank $\kappa$ by $I_\kappa(f)$, where $f(\cdot)$ is an integrand. For more details about Wiener-It\^{o} integrals and admissible functions $f(\cdot)$ one can refer to \cite{Ito, Maj}. The following result was obtained in \cite{New} for specific form of the integrand. Since the proof does not rely on the form of the integrand, this theorem can be easily generalized as follows.
\begin{theorem}{\rm \cite{New}}\label{cmb}
	For any $\kappa \in \mathbb{N}$ and an arbitrary positive $\varepsilon$ it holds
	\[\rho\left(I_\kappa(f),I_\kappa(f)+\varepsilon\right)\leq C\varepsilon^{b},\]
	where $b = 1$ if $\kappa < 3$ and $b = 1/\kappa$ if $\kappa \geq 3$.
\end{theorem}

Also, we will use the following result.
\begin{lemma}{\rm\cite{surfNew}}\label{dst}
	Let $f_1(\cdot)$ and $f_2(\cdot)$ be symmetric functions in $L_2({\mathbb{R}^d}),\,d\geq1$. Then,
	\begin{equation*}
	\rho\left(I_\kappa(f_1),I_\kappa(f_2) \right) \leq C\|f_1 - f_2\|^{\frac{1}{\kappa+1/2}}, \quad \text{if } \kappa \geq 3,
	\end{equation*}
	and
	\begin{equation*}
	\rho\left(I_\kappa(f_1),I_\kappa(f_2) \right) \leq C\|f_1 - f_2\|^{\frac{2}{3}}, \quad \text{if } \kappa < 3.
	\end{equation*}
\end{lemma}

Now we are ready to formulate the main result.

\begin{theorem}\label{th4}
	Let $H 	\mathrm{rank}\,G=\kappa \in \mathbb{N}$ and Assumptions~{\rm{\ref{ass1}}} and {\rm{\ref{ass2}}} hold for $\alpha\in(0, \frac{d-1}{\kappa})$.
	
	If $\tau \in \left(-\frac{d - \kappa\alpha}{2},0\right)$ then for any $%
	\varkappa<\frac{b}{2+b}\min\left(\frac{\alpha(d-1-\kappa\alpha)}{%
		d-1-(\kappa-1)\alpha},\varkappa_1\right)$ 
	\begin{equation*}
	{\rho}\left(X_{r, G},X_\kappa\right)=o (r^{-\varkappa}),\quad
	r\rightarrow \infty ,
	\end{equation*}
	where 
	$
	\varkappa_1:=\min\left(-2\tau,\frac{1}{\frac{1}{d-2\alpha}+ \dots +\frac{1}{%
			d-\kappa\alpha} +\frac{1}{%
			d-1-\kappa\alpha}}\right)
	$
	and $b$ is the parameter from Theorem~{\rm\ref{cmb}}.
	
	If $\tau=0$ then 
	\begin{equation*}
	{\rho}\left( X_{r, G},X_\kappa\right)=g^{\frac{2b}{2+b}}(r), \quad
	r\rightarrow \infty.
	\end{equation*}
\end{theorem}

\section{Proofs of the result from section \ref{sec2}}\label{sec3}

In this section the results stated in Section~2 are proved. In \cite{surfNew} Theorems \ref{th1}, \ref{th2}, \ref{th4} and Lemma~\ref{finint} were proved in the particular case $h_{sp}(\cdot) \equiv 1$. Since the proofs for the general setting are very similar to the ones in \cite{surfNew}, only parts of the proofs that differ are provided.

\begin{proof}[Proof of Theorem~\ref{th1}]
	
	By Remark~\ref{rem1} it holds $\mathrm{Var}X_{r, G} = \mathrm{Var}X_{r,\kappa } + \mathrm{Var}V_{r}$.
	
	By (\ref{hvar}) and (\ref{dint}) we get
	\[
	{\rm Var}\,V_r = c_r^2(d, \alpha)\sum_{j\geq \kappa + 1}\frac{C^2_j}{j!}
	\int\limits_{S(r)}\int\limits_{S(r)}\frac{L_0^j\left(\left\| x-y\right\| \right)}{\left\|
		x-y\right\|^{\alpha j}}h_{sp}\left(\frac{x}{\|x\|}\right)h_{sp}\left(\frac{y}{\|y\|}\right)\sigma(\mathrm{d}x)\, \sigma(\mathrm{d}y)
	\]
	\[
	\leq \frac{\kappa!^2c_2^{-\kappa}(d, \alpha)}{C^2_\kappa r^{2d-2-\kappa\alpha}L^\kappa(r)}\max\limits_{x\in S(1)}|h_{sp}(x)|^2\sum_{j\geq \kappa + 1}\frac{C^2_j}{j!}
	\int\limits_{S(r)}\int\limits_{S(r)}\frac{\left|L_0\left(\left\| x-y\right\| \right)\right|^j}{\left\|
		x-y\right\|^{\alpha j}} \sigma(\mathrm{d}x) \sigma(\mathrm{d}y)
	\]
	
	\[
	=\frac{\kappa!^2c_2^{-\kappa}(d, \alpha)r^{2d-2}}{C^2_\kappa r^{2d-2-\kappa\alpha}L^\kappa(r)}\max\limits_{x\in S(1)}|h_{sp}(x)|^2\sum_{j\geq \kappa + 1}\frac{C^2_j}{j!}
	\int\limits_{S(1)}\int\limits_{S(1)}\frac{\left|L_0\left(r\left\| x-y\right\| \right)\right|^j}{r^{\alpha j}\left\|
		x-y\right\|^{\alpha j}}\sigma(\mathrm{d}x)\sigma(\mathrm{d}y)
	\]
	\[
	= \frac{\kappa!^2|S(1)|^2 r^{2d-2}}{c_2^{\kappa}(d, \alpha)C^2_\kappa r^{2d-2-\kappa\alpha}L^\kappa(r)}\max\limits_{x\in S(1)}|h_{sp}(x)|^2\sum_{j\geq \kappa+1}\frac{C_j ^2}{j!}
	\int\limits_0^{2}\frac{\left|L_0\left(rz\right)\right|^j}{(rz)^{\alpha j}}\psi _{S(1)}(z)\mathrm{d}z.
	\]
	
	It follows from  $z^{-\alpha} \left|L_0\left(z\right)\right|\in[0,1],$ $z\ge 0,$  that
	\[
	{\rm Var}\,V_r \leq \frac{\kappa!^2|S(1)|^2 r^{2d-2-(\kappa+1)\alpha}}{c_2^{\kappa}(d, \alpha)C^2_\kappa r^{2d-2-\kappa\alpha}L^\kappa(r)}\max\limits_{x\in S(1)}|h_{sp}(x)|^2\sum_{j\geq \kappa+1}\frac{C_j ^2}{j!}
	\int\limits_0^{2}\frac{\left|L_0\left(rz\right)\right|^{\kappa + 1}}{z^{\alpha(\kappa + 1)}}\psi _{S(1)}(z)\mathrm{d}z.
	\]
	\[
	= \frac{\kappa!^2|S(1)|^2\left|L_0(r)\right|^\kappa}{c_2^{\kappa}(d, \alpha)C^2_\kappa L^\kappa(r)}\max\limits_{x\in S(1)}|h_{sp}(x)|^2\sum_{j\geq \kappa+1}\frac{C_j ^2}{j !}
	\int\limits_0^{2} z^{-\alpha\kappa} \frac{\left|L_0\left(rz\right)\right|^{\kappa}}{\left|L_0(r)\right|^{\kappa}}  \frac{\left|L_0\left(rz\right)\right|}{(rz)^{\alpha}}\psi _{S(1)}(z)\mathrm{d}z.
	\]
	
	Considering $|L_0(\cdot)|$ as a new slowly varying function, one can estimate the integral above using bounds (9) and (10) in \cite{surfNew} to obtain
	
	\begin{equation}\label{varb}
	{\rm Var}\,V_r \leq C\,\frac{\left|L_0(r)\right|^{\kappa}}{L^{\kappa}(r)}\left( r^{-\beta_1(d-1-\kappa\alpha-\delta)}
	+
	o\left(r^{-(\alpha-\delta)(1-\beta_1)}\right)\right),
	\end{equation}
	where $\delta$ is an arbitrary number in $(0, \min(\alpha, d - 1 - \kappa\alpha))$.
	
	Similar to $\mathrm{Var}\, V_r$ we obtain
	
	\[
	\mathrm{Var}X_{r,\kappa} = \frac{1}{c_2^{\kappa}(d, \alpha)L^\kappa(r)}
	\int\limits_{{S}(1)}\int\limits_{{S}(1
		)}\frac{L_0^\kappa\left(r\left\| x-y\right\| \right)}{\left\|
		x-y\right\|^{\alpha\kappa}} h_{sp}\left(x\right)h_{sp}\left(y\right)\mathrm{d}\sigma(x)\, \mathrm{d}\sigma(y).
	\]
	
	Let us consider a constant $c_h > \max\limits_{x}|h_{sp}(x)|$ and let $c_I := \int\limits_{{S}(1
		)}(h_{sp}(x) + c_h)\mathrm{d}\sigma(x)$. The function $\displaystyle\frac{h_{sp}(x) + c_h}{c_I}$ can be considered as a probability density function of some random variable $X$ defined on $S(1)$. Therefore, 
	\[
	\mathrm{Var} \, X_{r,\kappa} = \frac{c_I^2}{c_2^{\kappa}(d, \alpha)L^\kappa(r)}
	\int\limits_{{S}(1)}\int\limits_{{S}(1
		)}\frac{L_0^\kappa\left(r\left\| x-y\right\| \right)}{\left\|
		x-y\right\|^{\alpha\kappa}} \left(\frac{h_{sp}(x) + c_h}{c_I}\right)
	\]
	\[
	\times\left(\frac{h_{sp}(y) + c_h}{c_I}\right)\mathrm{d}\sigma(x)\, \mathrm{d}\sigma(y)
	- 2\frac{c_I|{S}(1)|}{c_2^{\kappa}(d, \alpha)L^\kappa(r)}
	\int\limits_{{S}(1)}\int\limits_{{S}(1
		)}\frac{L_0^\kappa\left(r\left\| x-y\right\| \right)}{\left\|
		x-y\right\|^{\alpha\kappa}} \left(\frac{h_{sp}(x) + c_h}{c_I}\right)
	\]
	\[
	\times\frac{ c_h}{|{S}(1)|}\mathrm{d}\sigma(x)\, \mathrm{d}\sigma(y)
	- \frac{|{S}(1)|^2}{c_2^{\kappa}(d, \alpha)L^\kappa(r)}
	\int\limits_{{S}(1)}\int\limits_{{S}(1
		)}\frac{L_0^\kappa\left(r\left\| x-y\right\| \right)}{\left\|
		x-y\right\|^{\alpha\kappa}} \left(\frac{ c_h}{|{S}(1)|}\right)^2\mathrm{d}\sigma(x)\, \mathrm{d}\sigma(y) 
	\]
	\[
	= \frac{c_I^2}{c_2^{\kappa}(d, \alpha)L^\kappa(r)}\mathrm{E} \left[\frac{L_0^\kappa\left(r\left\| X_1-X_2\right\| \right)}{\left\|
		X_1-X_2\right\|^{\alpha\kappa}}\right]
	 -2\frac{c_I|{S}(1)|c_h}{c_2^{\kappa}(d, \alpha)L^\kappa(r)}\mathrm{E} \left[\frac{L_0^\kappa\left(r\left\| X_1-Y_1\right\| \right)}{\left\|
		X_1-Y_1\right\|^{\alpha\kappa}}\right]
	\]
	\[ - \frac{|{S}(1)|^2c_h^2}{c_2^{\kappa}(d, \alpha)L^\kappa(r)}\mathrm{E} \left[\frac{L_0^\kappa\left(r\left\| Y_1- Y_2\right\| \right)}{\left\|
		Y_1 - Y_2\right\|^{\alpha\kappa}}\right],
	\]
	where $X_1,\,X_2$ are random variables distributed on $S(1)$ with the probability density $\displaystyle\frac{h_{sp}(x) + c_h}{c_I}$ and $Y_1,\,Y_2$ are uniformly distributed on $S(1)$ random variables. 
	
	Let $Z_1 := \left\| X_1-X_2\right\|$, $Z_2 := \left\| X_1-Y_1\right\|$, $Z_3 := \left\| Y_1-Y_2\right\|$ and denote their corresponding probability densities as $\bar{\psi}_{S(1)}(\cdot)$, $\tilde{\psi}_{S(1)}(\cdot)$ and $\psi_{S(1)}(\cdot)$. Thus
	\[
	\mathrm{Var} \, X_{r,\kappa} = \frac{c_I^2}{c_2^{\kappa}(d, \alpha)L^\kappa(r)}\int\limits_0^{2} z^{-\alpha\kappa}L_0^\kappa\left(rz\right)\bar{\psi}_{S(1)}(z)\mathrm{d}z -2\frac{c_I|{S}(1)|c_h}{c_2^{\kappa}(d, \alpha)L^\kappa(r)}
	\]
	\[
	\times\int\limits_0^{2} z^{-\alpha\kappa}L_0^\kappa\left(rz\right)\tilde{\psi} _{S(1)}(z)\mathrm{d}z
	- \frac{|{S}(1)|^2c_h^2}{c_2^{\kappa}(d, \alpha)L^\kappa(r)}\int\limits_0^{2} z^{-\alpha\kappa}L_0^\kappa\left(rz\right){\psi}_{S(1)}(z)\mathrm{d}z.
	\]
	
	If $\alpha \in (0,(d-1)/\kappa)$ then by asymptotic properties of integrals of slowly varying functions (see Theorem~2.7 in \cite{sen}) we get

	\begin{equation}\label{varf}
	\mathrm{Var} \, X_{r,\kappa}=\frac{(\bar{c}_1(\kappa,\alpha) +\tilde{c}_1(\kappa,\alpha)+{c}_1(\kappa,\alpha))}{c_2^{\kappa}(d, \alpha)}\, \frac{L^{\kappa}_0(r)}{L^{\kappa}(r)}(1+o(1)),
	\quad r\to \infty,
	\end{equation}	
	where $\bar{c}_1(\kappa,\alpha):=c_I^2\int\limits_0^{2}\frac{\bar{\psi}_{S(1)}(z)}{z^{\alpha\kappa}}\mathrm{d}z$, $\tilde{c}_1(\kappa,\alpha):=-2c_I|{S}(1)|c_h\int\limits_0^{2}\frac{\tilde{\psi}_{S(1)}(z)}{z^{\alpha\kappa}} \mathrm{d}z$, and	${c}_1(\kappa,\alpha) := {-\displaystyle|{S}(1)|^2c_h^2\int\limits_0^{2}\frac{{\psi}_{S(1)}(z)}{z^{\alpha\kappa}} \mathrm{d}z}$.

	By (\ref{dint}) it holds
	\[\bar{c}_1(\kappa,\alpha) +\tilde{c}_1(\kappa,\alpha)+{c}_1(\kappa,\alpha) = \int\limits_{{S}(1)}\int\limits_{{S}(1
		)}\frac{(h_{sp}(x) + c_h)(h_{sp}(y) + c_h)}{\left\|
		x-y\right\|^{\alpha\kappa}}\mathrm{d}\sigma(x)\, \mathrm{d}\sigma(y) \]
	\[
	-2c_h\int\limits_{{S}(1)}\int\limits_{{S}(1
		)}\frac{(h_{sp}(x) + c_h)}{\left\|
		x-y\right\|^{\alpha\kappa}}\mathrm{d}\sigma(x)\, \mathrm{d}\sigma(y) -c_h^2\int\limits_{{S}(1)}\int\limits_{{S}(1
		)}\frac{1}{\left\|
		x-y\right\|^{\alpha\kappa}}\mathrm{d}\sigma(x)\, \mathrm{d}\sigma(y)\]
	\[
	\leq \max\limits_{x}|h_{sp}(x)|^2|S(1)|^2\int\limits_0^{2}\frac{{\psi}_{S(1)}(z)}{z^{\alpha\kappa}} \mathrm{d}z < \infty.
	\]
	By (\ref{varb}) and (\ref{varf}), one can choose $\beta_1=1/2$ and make $\delta$ arbitrary close to 0 to obtain
	\[
	\lim_{r\to \infty }\frac{\mathrm{Var}\, V_r}{\mathrm{Var}\, X_{r, G}}=0\quad \mbox{and}\quad
	\lim_{r\to \infty }\frac{\mathrm{Var}\, X_{r, G}}{\mathrm{Var}\, X_{r,\kappa}} = 1.
	\]
	
	Thus
	\[
	\lim_{r\to \infty }\, \mathrm{E}\left(\frac{X_{r, G}}{\sqrt{ \mathrm{Var} \ X_{r, G}}}-\frac{X_{r,\kappa}}{\sqrt{ \mathrm{Var} \ X_{r,\kappa}}}\right)^2=\lim_{r\to \infty }\frac{\mathrm{E}\left(V_r+\left(1-\sqrt{\frac{\mathrm{Var} X_{r, G}}{\mathrm{Var} X_{r,\kappa}}}\right)X_{r,\kappa}\right)^2}{\mathrm{Var} X_{r, G}} =0,\]
	and
	\[
	\lim_{r\to \infty }\, \mathrm{E}\left(\frac{X_{r, G}}{\sqrt{ \mathrm{Var} \ X_{r,\kappa}}}-\frac{X_{r,\kappa}}{\sqrt{ \mathrm{Var} \ X_{r,\kappa}}}\right)^2=\lim_{r\to \infty }\frac{\mathrm{E}\left(V_r\sqrt{\frac{\mathrm{Var} X_{r, G}}{\mathrm{Var} X_{r,\kappa}}}\right)^2}{\mathrm{Var} X_{r, G}} =0\]
	which completes the proof. \qedhere
	
\end{proof}

To obtain Lemma~\ref{finint} a result analogous to the average decay rate of the Fourier transforms of surface measures in \cite{Ios} is needed. Therefore we prove the following lemma.

\begin{lemma}\label{avdw} For sufficiently large $r$ it holds
	\begin{equation}\label{adw}
	\int\limits_{S(1)}|K(\omega r)|^2\,\mathrm{d}\omega\leq Cr^{1-d}.
	\end{equation}
\end{lemma}

\begin{proof}
	Using Theorem 1 in \cite[p. 322]{Stein} we obtain $|K(x)|\leq  C\|x\|^{\frac{1-d}{2}}.$
	
	It follows
	\[\int\limits_{S(1)}|K(\omega r)|^2\,\mathrm{d}\omega\leq \int\limits_{S(1)}Cr^{1-d}\,\mathrm{d}\omega\leq Cr^{1-d}.\qedhere\]
	
\end{proof}

Now, we can prove Lemma~\ref{finint} using bound (\ref{adw}) analogously to Lemma~2 in \cite{surfNew}.

\begin{proof}[Proof of Lemma~\ref{finint}]
	For $\kappa= 1$ it holds $d-\tau_1>1$. By the properties of the Fourier transform $|K(\lambda)| 
	\leq C|S(1)|$ for all $\lambda \in \mathbb{R}^d$. Therefore, using integration formula for polar coordinates we get	
	\[
	\int\limits_{\mathbb{R}^{d}}|K(\lambda)|^2  \frac{\mathrm{d}\lambda}{\left\| \lambda\right\| ^{d-\tau_1}}= \int\limits_{0}^{\infty}r^{d-1}\int\limits_{S(1)}\frac{|K(\omega r)|^2}{r^{d-\tau_1}}\,\mathrm{d}\omega\mathrm{d}r = \int\limits_{0}^{r_0}r^{d-1}\int\limits_{S(1)}\frac{|K(\omega r)|^2}{r^{d-\tau_1}}\,\mathrm{d}\omega\mathrm{d}r \]
	\[+ \int\limits_{r_0}^{\infty}r^{d-1}\int\limits_{S(1)}\frac{|K(\omega r)|^2}{r^{d-\tau_1}}\,\mathrm{d}\omega\mathrm{d}r
	\leq C|S(1)|^2\int\limits_{0}^{r_0}\frac{r^{d-1}\mathrm{d}r}{r^{d-\tau_1}} + \int\limits_{r_0}^{\infty}r^{d-1}\int\limits_{S(1)}\frac{|K(\omega r)|^2}{r^{d-\tau_1}}\,\mathrm{d}\omega\mathrm{d}r.\]
	By Lemma~\ref{avdw} we obtain
	\[
	\int\limits_{\mathbb{R}^{d}}|K(\lambda)|^2  \frac{\mathrm{d}\lambda}{\left\| \lambda\right\| ^{d-\tau_1}} \leq C |{S}(1)|^2\int\limits_{0}^{r_0}\frac{\mathrm{d}r}{r^{1-\tau_1}} +C\int\limits_{r_0}^{\infty}\frac{r^{-d+1}}{r^{1-\tau_1}}\mathrm{d}r\]
	\[
	= C|{S}(1)|^2\int\limits_{0}^{r_0}\frac{\mathrm{d}r}{r^{1-\tau_1}} +C\int\limits_{r_0}^{\infty}\frac{\mathrm{d}r}{r^{d-\tau_1}} <\infty.
	\]
	
	For $\kappa> 1$ one can obtain (\ref{finv}) by the recursive estimation routine and the change of variables $u=\lambda_{\kappa-1}+\lambda_\kappa$ and $\tilde{\lambda}_{\kappa-1}={\lambda_{\kappa-1}}/{\left\| u\right\|}$, see \cite{surfNew} for details.	
\end{proof}

The remaining theorems consider the random variables $X_{r,\kappa}$ and $X_\kappa$. As it was mentioned before, in \cite{surfNew} these theorems were proved for the particular case $h_{sp}(\cdot) \equiv 1$. Up to multiplication by a constant, $X_\kappa$ in this paper is identical to its counterpart $X_\kappa(\Delta)$ in \cite{surfNew} and does not depend on the weight function $h_{sp}(\cdot)$. If $X_{r,\kappa}$ is represented as the multiple Wiener-It\^{o} stochastic integral the function $h_{sp}(\cdot)$ appears only in the function $K(\cdot)$.  Note that Lemma~\ref{finint} and bound~(\ref{adw}) for $K(\cdot)$ are the same as the corresponding results  for $\mathcal{K}(\cdot)$ in \cite{surfNew}.  Since Lemma~\ref{finint} and bound~(\ref{adw}) are the only results concerning $K(\cdot)$ that are required, to prove theorems \ref{th2} and \ref{th4} one can follow the corresponding proofs in \cite{surfNew}. Therefore, only the key details in the proofs for Theorems \ref{th2} and \ref{th4} are provided.

\begin{proof}[Proof of Theorem~\ref{th2}]
	Using It\^{o} formula (2.3.1) in \cite{LeoLT} we obtain
	\[
	\int\limits_{S(r) } H_\kappa(\eta (x))h_{sp}\left(\frac{x}{\|x\|}\right)\sigma(\mathrm{d}x)=\int\limits_{S(r) }{\int\limits_{\mathbb{R}^{d\kappa}}}^{\prime}e^{i<\lambda _1+\cdots +\lambda
		_\kappa,x>}\]
	\[\times \prod\limits_{j=1}^\kappa\sqrt{f(\|\lambda _j\|)} W(\mathrm{d}\lambda
	_j)h_{sp}\left(\frac{x}{\|x\|}\right)\sigma(\mathrm{d}x).
	\]
	As $h_{sp}(\cdot)$ is bounded and $\prod\limits_{j=1}^\kappa\sqrt{f(\|\lambda _j\|)}\in L_2(\mathbb{R}^{d\kappa})$ then the stochastic Fubini theorem, see Theorem~5.13.1 in \cite{pec}, can be used to interchange the integrals. It results in
	\begin{equation}\label{spl}
	X_{r,\kappa}\stackrel{\mathcal{D}}{=}
	{\int\limits_{\mathbb{R}^{d\kappa}}}^{\prime}\frac{K(\lambda _1+\cdots +\lambda _\kappa)Q_r(\lambda _1,\ldots ,\lambda _\kappa)W(\mathrm{d}\lambda
		_1)\ldots W(\mathrm{d}\lambda _\kappa)}{\left\| \lambda _1\right\| ^{(d-\alpha )/2}\cdots
		\left\| \lambda _\kappa\right\| ^{(d-\alpha )/2}},
	\end{equation}
	where
	\[
	Q_r(\lambda _1,\ldots ,\lambda _\kappa): =r^{\kappa(\alpha-d)/2}L^{-\kappa/2}(r)\
	c_2^{-\kappa/2}(d,\alpha)  \left[ \prod\limits_{j=1}^\kappa\left\| \lambda _j\right\| ^{d-\alpha}f\left( \frac{\left\| \lambda _j\right\| }r\right) \right] ^{1/2}.
	\]

	Using Lemma~\ref{finint} instead of its counterpart in \cite{surfNew}, the rest of the theorem can be proven analogously to Theorem~2 in \cite{surfNew}. 
\end{proof}

\begin{proof}[Proof of Theorem~\ref{th4}]
	Since $X_{r, G} = X_{r,\kappa}+V_r$, applying Chebyshev's inequality and Lemma~\ref{lem1} to $X=X_{r, \kappa}$, $Y=V_r$ and $Z=X_{\kappa},$ we get
	\begin{equation}\label{midrate}
	{\rho}\left( X_{r, G},X_\kappa\right)={\rho}\left( X_{r,\kappa}+V_r,X_\kappa\right)		
	\leq {\rho}\left(X_{r, \kappa},X_\kappa\right)+{\rho}\left(X_{\kappa}+\varepsilon, X_\kappa\right)+\varepsilon^{-2}{\rm Var}\,V_r.
	\end{equation}
	
	It follows from Theorem~\ref{cmb} that 
	\[{\rho}\left(X_\kappa+\varepsilon,X_\kappa\right)\le C\varepsilon^b.\]

	By (\ref{varb}) it holds
	\[
	{\rm Var}\,V_r \leq C\,\frac{\left|L_0(r)\right|^{\kappa}}{L^{\kappa}(r)}\left( r^{-\beta_1(d-1-\kappa\alpha-\delta)}
	+
	o\left(r^{-(\alpha-\delta)(1-\beta_1)}\right)\right).
	\]
	Since, by Remark~\ref{eql}, $L_0(\cdot) \sim L(\cdot)$, we can replace $L_0(\cdot)$ by $L(\cdot)$ in the above estimate. Thus, choosing $\beta_1=\frac{\alpha}{d-1-(\kappa - 1)\alpha}$ to minimize the upper bound we get
	\[
	{\rm Var}\,V_r \leq C r^{-\frac{\alpha(d-1-\kappa\alpha)}{d-1-(\kappa - 1)\alpha}+\delta}.
	\]

	Choosing $\varepsilon:=r^{-\frac{\alpha(d-1-\kappa\alpha)}{(2+b)(d-1-(\kappa - 1)\alpha)}}$ to minimize the second term in estimate (\ref{midrate}) we obtain
	\[
	{\rho}\left(X_{r, G},X_\kappa\right)\le {\rho}\left(X_{r, \kappa},X_\kappa\right)+C\,r^{\frac{-b\alpha(d-1-\kappa\alpha)}{(2+b)(d-1-(\kappa - 1)\alpha)}+\delta}.
	\]
	
	By (\ref{spl}) 
	\[X_{r,\kappa}\stackrel{\mathcal{D}}{=}
	{\int\limits_{\mathbb{R}^{d\kappa}}}^{\prime}\frac{K(\lambda _1+\cdots +\lambda _\kappa)Q_r(\lambda _1,\ldots ,\lambda _\kappa)W(\mathrm{d}\lambda
		_1)\ldots W(\mathrm{d}\lambda _\kappa)}{\left\| \lambda _1\right\| ^{(d-\alpha )/2}\cdots
		\left\| \lambda _\kappa\right\| ^{(d-\alpha )/2}}.\]
	Therefore, ${\rho}\left(X_{r, \kappa},X_\kappa\right)$ is the Kolmogorov distance between two multiple Wiener-It\^{o} integrals of the rank~$\kappa$. Using Lemma~\ref{dst} we get
	\[
	\rho\left(X_{r, \kappa},X_{\kappa} \right) \leq C\left[\ \int\limits_{\mathbb{R}^{\kappa d}}\frac{|K(\lambda _1+ \dots +\lambda _\kappa)|^2\left(Q_r(\lambda_1,\dots,\lambda_\kappa)-1\right)^2\mathrm{d}\lambda
		_1 \dots \,\mathrm{d}\lambda _\kappa}{\left\| \lambda _1\right\| ^{d-\alpha}\dots\left\| \lambda _\kappa\right\| ^{d-\alpha}}\right]^{\frac{a}{2+a}}.
	\]
	
	Using Lemma~\ref{finint} and the bound (\ref{adw}) instead of Lemma~2 and (2) in \cite{surfNew} respectively, the rest of the proof is analogous to Theorem~5 in \cite{surfNew}. \qedhere
	
\end{proof}

\section{Simulation studies}\label{sec4}

In this section we present some simulation studies to confirm the obtained results. We also investigate random fields on cube surfaces and demonstrate that the rate of convergence for cube surfaces is similar to the rate of convergence for spheres. These simulations support the hypothesis that the results of this paper should hold for more general surfaces.

The theoretical results in Sections~\ref{sec2} and \ref{sec3} were derived for integral functionals. However, computer simulations are possible only in discrete spaces. Thus, we simulate the functionals for sums instead of integrals. Recent results in \cite{OleAlo} show that additive functionals can be used to provide reliable approximations of the integral functionals.

All simulations were done by using parallel computing on the NCI's high-performance computer Raijin. The codes used for simulations and examples in this article are available in the folder "Research materials" from  \url{https://sites.google.com/site/olenkoandriy/}

We consider the Cauchy random field defined on the spheres $S(r),\, r>0$. The covariance function of this field is
$
\text{\rm{B}}(x) = (1 + \|x\|^2)^{-\frac{\alpha}{2}},$ $x\in S(r)$, where $\alpha$ is a long-range dependence parameter from Assumption~\ref{ass1}. By \cite[p.293]{Doh} the spectral density of this field is
\[
f_c(\|\lambda\|) = \frac{\|\lambda\|^{\frac{\alpha-d}{2}}K_{\frac{d-\alpha}{2}}(\|\lambda\|)}{\pi^{d/2}2^{\frac{\alpha-d}{2}}\Gamma\left(\frac{\alpha}{2}\right)} = c_2(d,\alpha )\|\lambda\|^{\alpha -d} \times \frac{2^{\frac{d+\alpha}{2}}}{\Gamma((d - \alpha)/2)}\|\lambda\|^{\frac{d-\alpha}{2}}K_{\frac{d-\alpha}{2}}(\|\lambda\|),
\]
where $K_{\frac{d-\alpha}{2}}(\cdot)$ is the modified Bessel function of the third kind.

By (10.30.2) in \cite{Nist} $K_{\nu}(z) \sim \frac{1}{2}\Gamma(\nu)(\frac{1}{2}z)^{-\nu},\, z\rightarrow 0$. Thus, one can see that \[L_c(\|\lambda\|):= \frac{2^{\frac{d+\alpha}{2}}}{\Gamma((d - \alpha)/2)}\|\lambda\|^{\frac{\alpha-d}{2}}K_{\frac{d-\alpha}{2}}\left(\frac{1}{\|\lambda\|}\right) 
\sim 2^{d-1},\, \|\lambda\| \rightarrow \infty,
\] 
is a slowly varying function that satisfies Assumption~\ref{ass2}. Hence, Theorem~\ref{th4} holds for the Cauchy random fields.

Random fields were simulated by using the R package \textrm{'RandomFields'}, see \cite{Schlather}. Figure~\ref{fig1} provides two examples of the simulated Cauchy random field that were obtained. 

\begin{figure}[h]
	\begin{minipage}{7cm}
		\includegraphics[width=0.8\linewidth,height=0.8\linewidth]{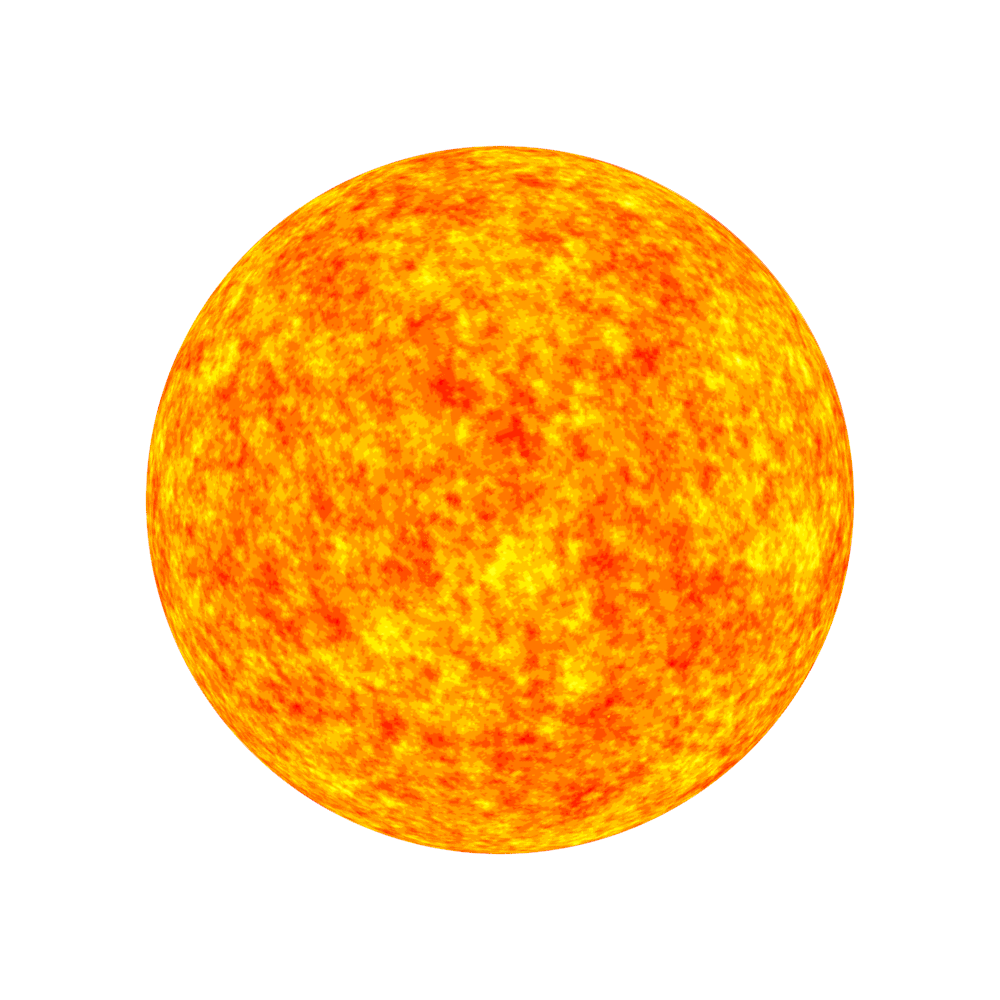}
	\end{minipage} \
	\begin{minipage}{7cm}
		\includegraphics[width=0.8\linewidth,height=0.8\linewidth]{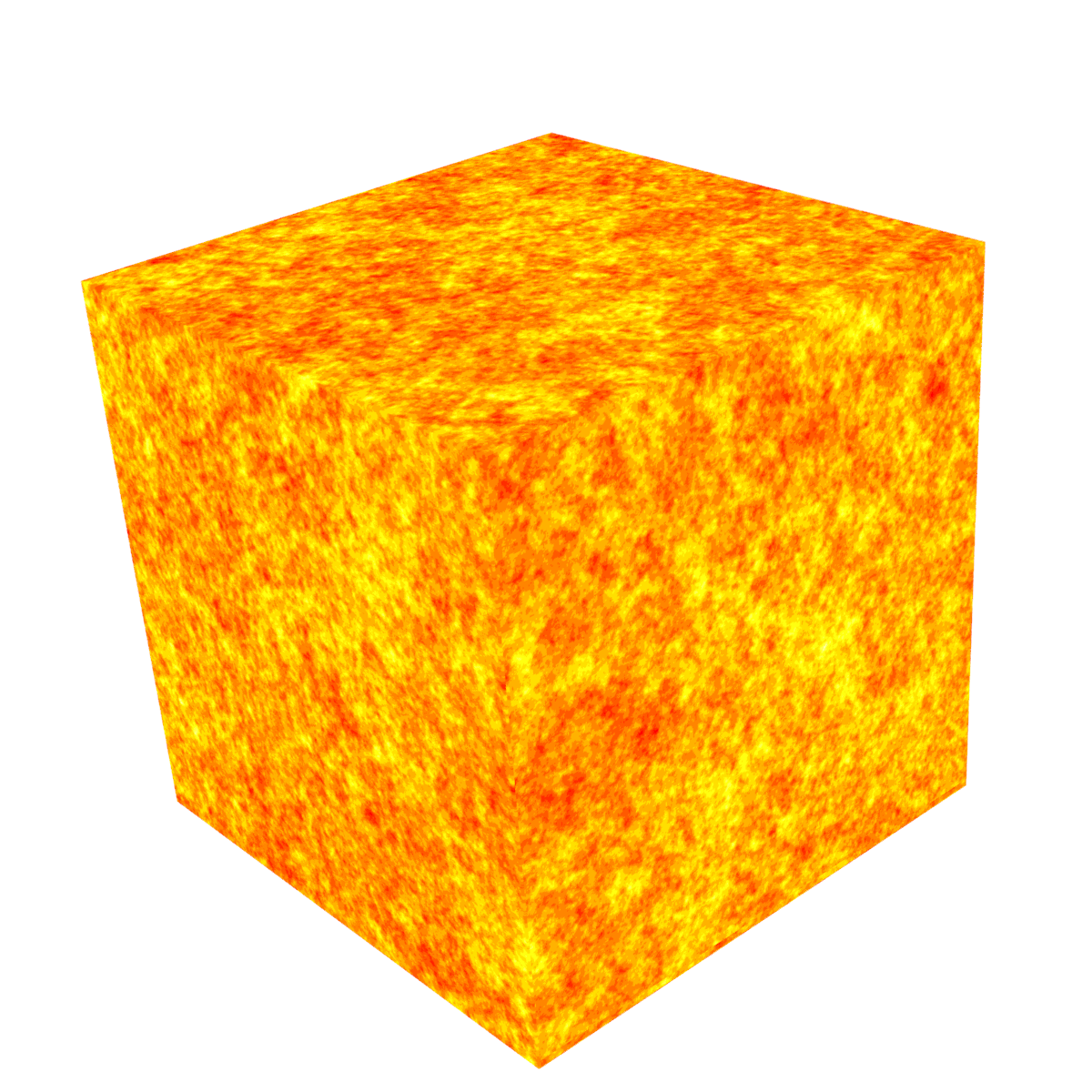}
	\end{minipage}
	\caption{Realizations of Cauchy random fields with $\alpha = \frac{2}{3}$ on the sphere and the surface of the cube.}
	\label{fig1}
\end{figure}

To verify the results of Sections~\ref{sec2} and \ref{sec3} we simulated the random variables 
\begin{equation}\label{vardisc}
X_{r, G} = \frac{\kappa!c_2^{-\kappa/2}(d,\alpha )}{C_\kappa r^{d-1-\frac{\kappa\alpha}{2}}L_c^{%
		\frac{\kappa}{2}}(r)}\frac{|S(r)|}{|P(r)|}\sum\limits_{p \in P(r)}G(\eta(p))h_{sp}\left(\frac{p}{\|p\|}\right),\, r \in [r_{min}, r_{max}],
\end{equation}
where $P(r) := \{p \in S(r)\}$ is a discrete set of uniformly distributed points in $S(r)$, $|S(\cdot)|$ is the surface area of the sphere, $|P(\cdot)|$ is the number of points in the set $P(\cdot)$, and $r_{min}$, $r_{max}$ denote the minimum and maximum  values of the considered radius.

For the limit random variable $X_\kappa$ we considered the random variable $X_{R, G}$, where $R>>r_{max}$.

Let $d = 3$, $\alpha = 2/3$, $G(y) = H_2(y) = y^2 -1$, $r_{min} = 200$, and $r_{max} = 3000$. Then the variables in (\ref{vardisc}) take the form
\begin{equation}\label{sph2}
X_{r, 2} = \frac{\pi^{3/2}\Gamma\left(\frac{1}{3}\right)}{2^{7/6}r^{1/{6}}K_{\frac{7}{6}}(1/r)}\frac{|S(r)|}{|P(r)|}\sum\limits_{p \in P(r)}((\eta(p))^2 - 1)h_{sp}\left(\frac{p}{\|p\|}\right),\, r \in [200, 3000].
\end{equation}

Let $h_{sp}^{s}(\theta, \phi) := 1.2 + 0.2\sin(5\theta)\sin(5\phi)$ and $h_{sp}^{c}(\theta, \phi) := 2 + \cos(3\theta)$, where $\theta$ and $\phi$ are spherical coordinates. These functions embed the sphere of radius 1 and smoothly vary in different directions. We present two examples where we used $h_{sp}^{s}(\cdot)$ as a weighted function in the case of spheres and $h_{sp}^{c}(\cdot)$ as a weighted function in the case of cubes. See Figure~\ref{fig2} for the visualization of these functions.

\begin{figure}[h]
	\begin{minipage}{7cm}
		\includegraphics[width=0.8\linewidth,height=0.8\linewidth]{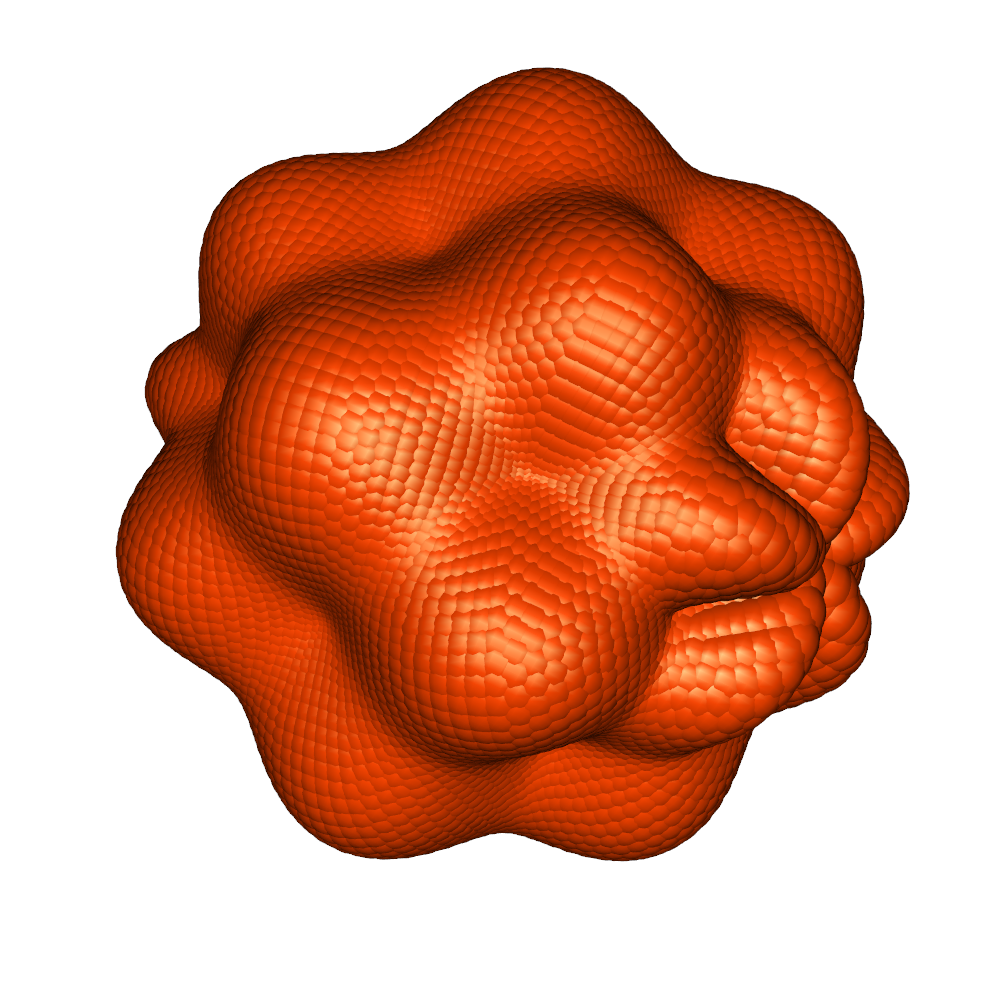}
	\end{minipage} \
	\begin{minipage}{7cm}
		\includegraphics[width=0.8\linewidth,height=0.8\linewidth]{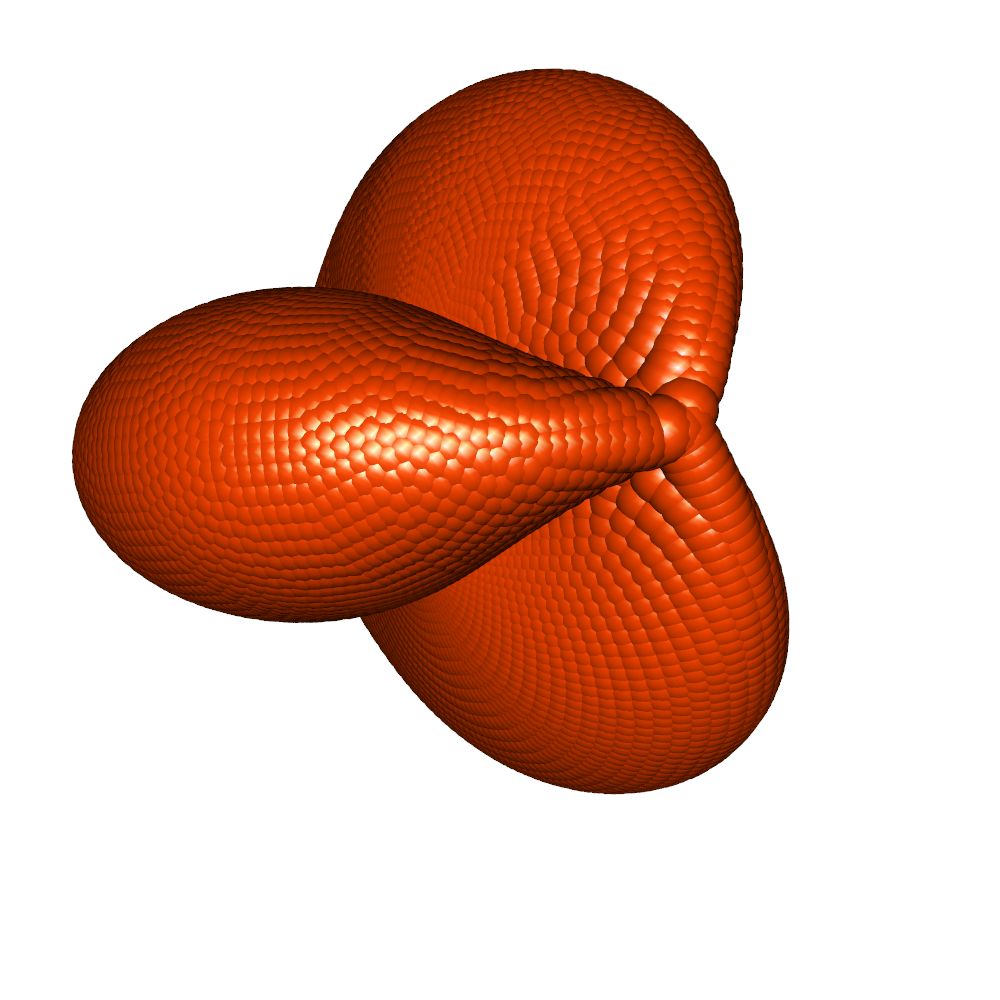}
	\end{minipage}
	\caption{Weighted functions $h_{sp}^{s}$ (left) and $h_{sp}^{c}$ (right) used in the simulations.}
	\label{fig2}
\end{figure}

The sets $P(r),\,  r >0$, were created by using the algorithm in \cite{Des} and setting the number of points $|P(r)| \sim r^2$. The random field $\eta(\cdot)$ was simulated at the points $P(r), r \in [200, 3000],$ and at the points $P(R)$. Each random variable (\ref{sph2}) was simulated 2000 times. Using the simulated values, the Kolmogorov distance between $X_{r, 2}$ and $X_{R, 2}$ was computed by the default Kolmogorov-Smirnov function \textrm{'ks.test'} in R. We present the results in two figures, obtained by repeating the simulation steps above 100 times. The first figure shows the box plots of the distances between $X_{r, 2}$ and $X_{R, 2}$. The second figure shows the corresponding box plots of the distance's logarithms. The first figure displays actual convergence of the Kolmogorov distance for increasing $r$. The second figure displays the changes in the power index of the rate of convergence.

For $R = 4000$ and $h_{sp}(\cdot) \equiv 1$ Figure~\ref{fig3} presents the obtained results for the unweighted case on spheres.

\begin{figure}[h]
	\begin{minipage}{7cm}
		\includegraphics[width=1\linewidth,height=0.8\linewidth]{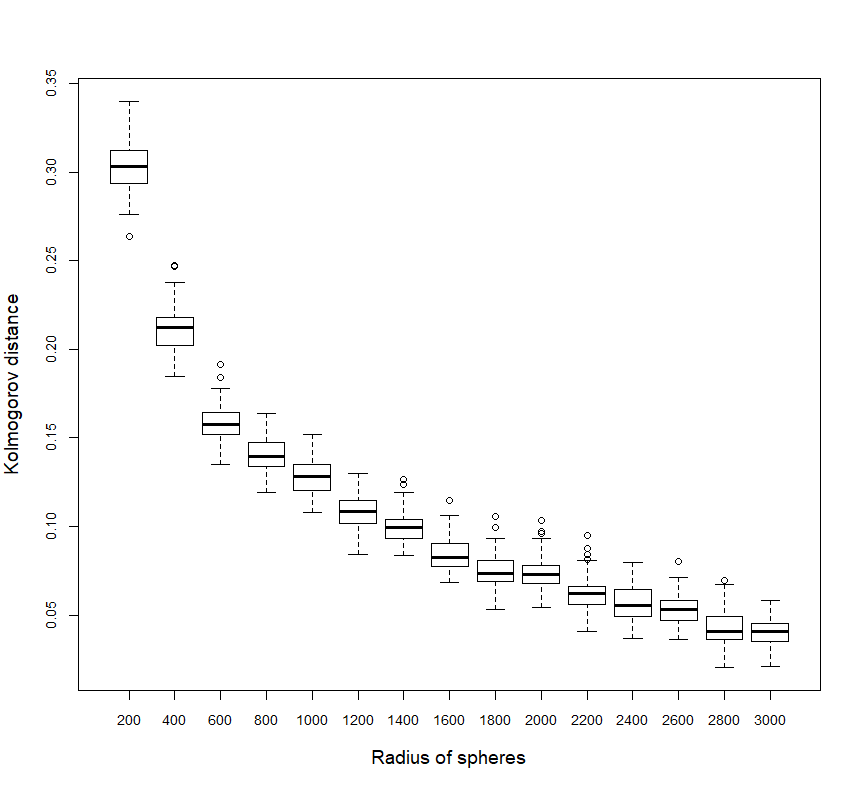}
	\end{minipage} \
	\begin{minipage}{7cm}
		\includegraphics[width=1\linewidth,height=0.8\linewidth]{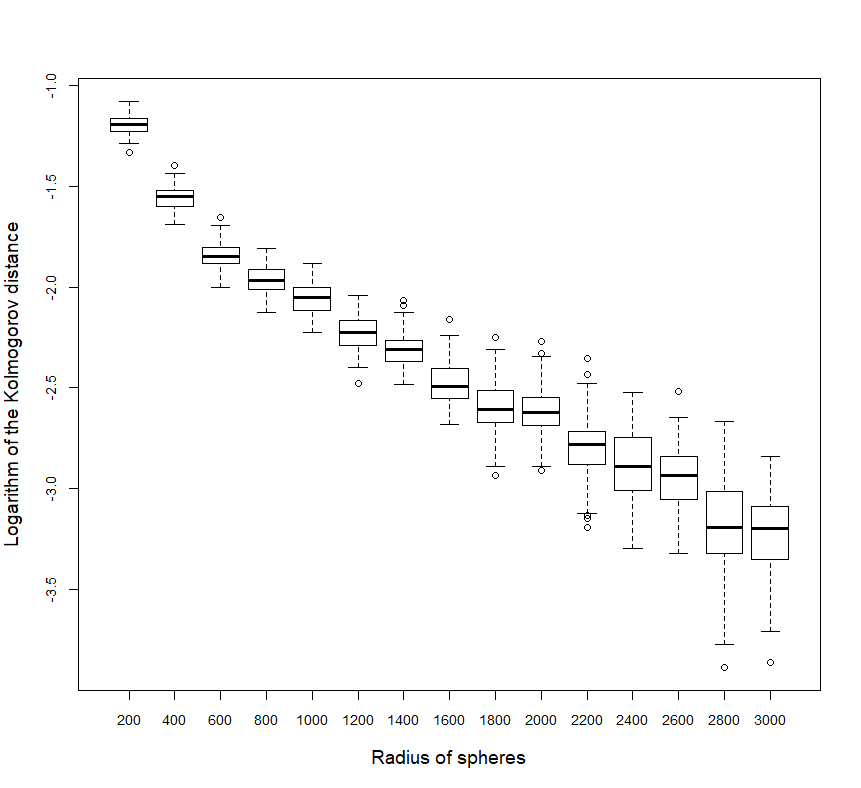}
	\end{minipage}
	\caption{Simulation results for spheres with $R = 4000$ and $h_{sp}(\cdot)\equiv 1$.}
	\label{fig3}
\end{figure}

In the weighted case $h_{sp}(\cdot) = h_{sp}^{s}(\cdot)$ the obtained results are presented in Figure~\ref{fig4}. 

\begin{figure}[h]
	\begin{minipage}{7cm}
		\includegraphics[width=1\linewidth,height=0.8\linewidth]{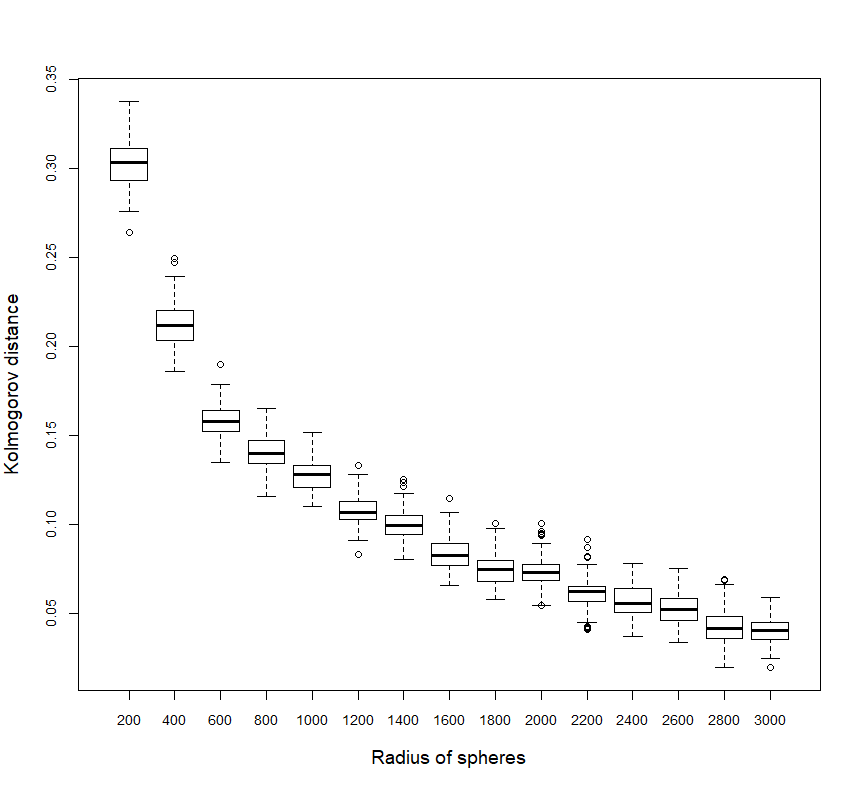}
	\end{minipage} \
	\begin{minipage}{7cm}
		\includegraphics[width=1\linewidth,height=0.8\linewidth]{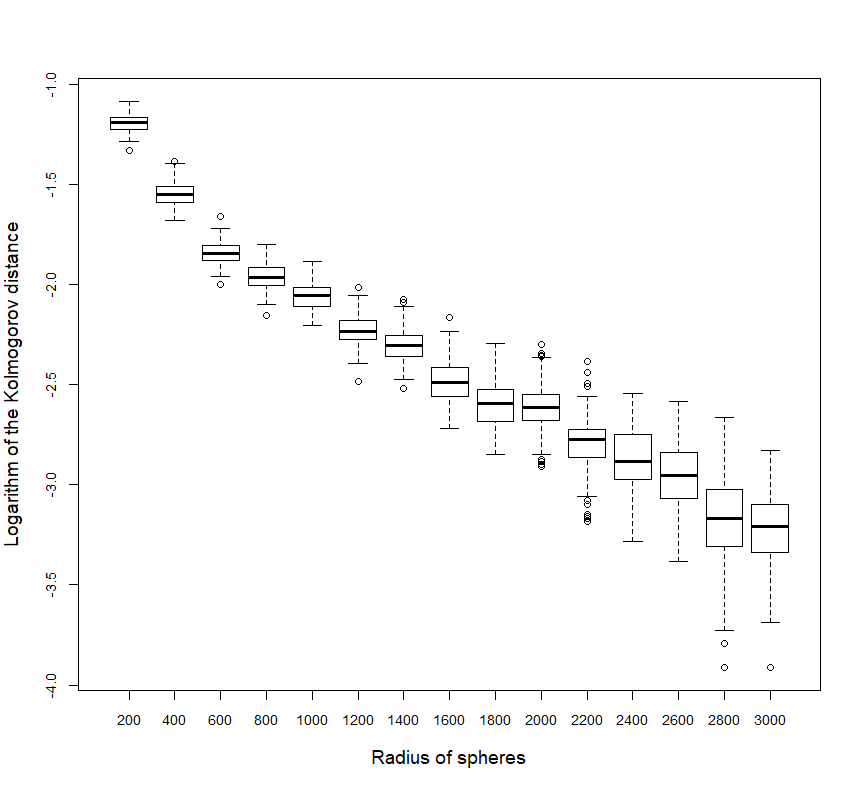}
	\end{minipage}
	\caption{Simulation results for spheres with $R = 4000$ and $h_{sp}(\cdot) = h_{sp}^{s}(\cdot)$.}
	\label{fig4}
\end{figure}

As one can see, Figures~\ref{fig3} and \ref{fig4} are quite similar which indicates that un-weighted and weighted functionals have the same rate of convergence. This observation agrees with Theorem~\ref{th4} according to which the weighted function has no effect on the rate of convergence. If the upper bound for the convergence rate in Theorem~\ref{th4} was sharp, one would expect to see the means of Kolmogorov distance's logarithms converge to a fixed value.  Instead, we can see that the means form a declining slope, which might suggest that the real rate of convergence is a power function of the radius $r$. By fitting the linear regression model we obtained the equation 
\begin{equation}\label{pwr}
\log(\rho\left(X_{r, 2},X_{R, 2}\right)) \approx -1.512 -0.000576r
\end{equation} in the unweighted case and a very similar equation in the weighted case.

Similar simulations  were also conducted for cube surfaces. We considered 3-dimensional cube surfaces $SQ(r)$ with the centre at the origin and the side lengths $2r,\, r>0$. Thus, in (\ref{sph2}) we replaced $S(r)$ by $SQ(r)$ and $P(r)$ by $PQ(r) := \{p \in SQ(r)\}$. Each $PQ(r)$ was set by creating the equidistant grid for each side of the cube and, therefore, $\left|PQ(r)\right| \sim r^2$. To approximate integrals (\ref{intf}) by sums (\ref{vardisc}) with the equal precision in both sphere and cube cases we created approximately $6/\pi$ times more points on cubes than on spheres of the same radius. Other steps remained unchanged.   

In the unweighted case $h_{sp}(\cdot) \equiv 1$ the obtained results for surfaces of cube are presented in Figure~\ref{fig5}.

\begin{figure}[h]
	\begin{minipage}{7cm}
		\includegraphics[width=1\linewidth,height=0.8\linewidth]{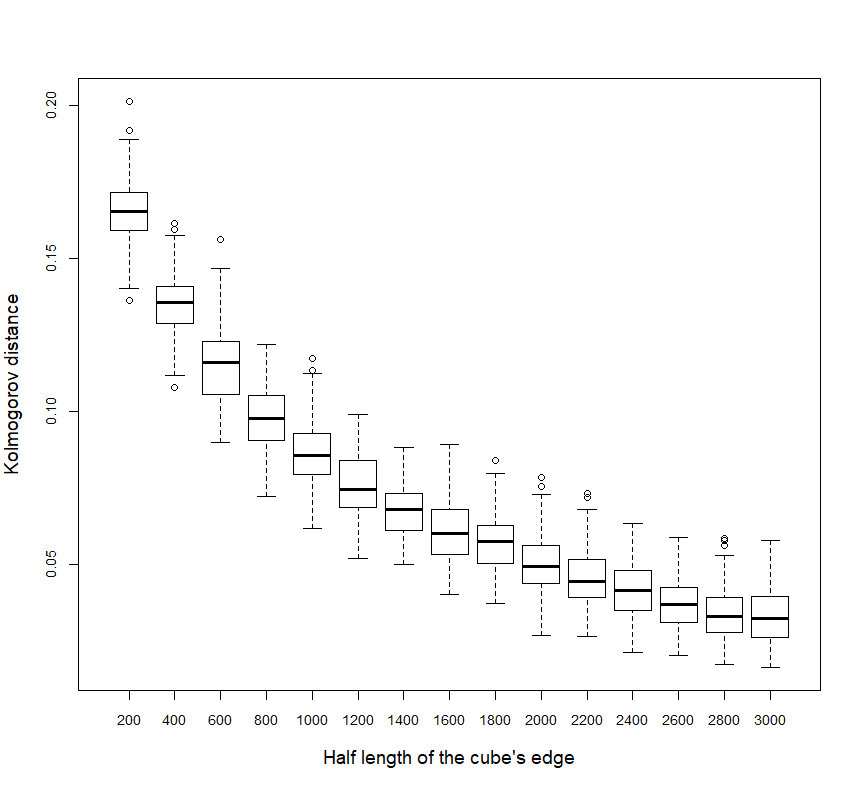}
	\end{minipage} \
	\begin{minipage}{7cm}
		\includegraphics[width=1\linewidth,height=0.8\linewidth]{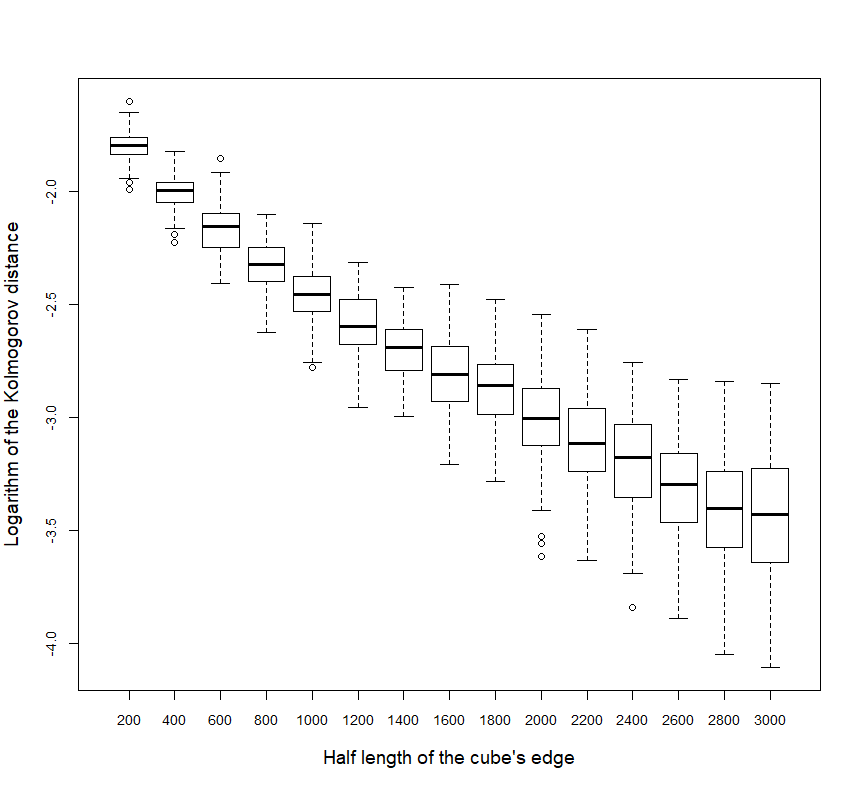}
	\end{minipage}
	\caption{Simulation results for cubes with $R = 4000$ and $h_{sp}(\cdot)\equiv 1$.}
	\label{fig5}
\end{figure}

In the weighted case $h_{sp}(\cdot) = h_{sp}^{c}(\cdot)$ we obtained the results presented in Figure~\ref{fig6}. 

\begin{figure}[h]
	\begin{minipage}{7cm}
		\includegraphics[width=1\linewidth,height=0.8\linewidth]{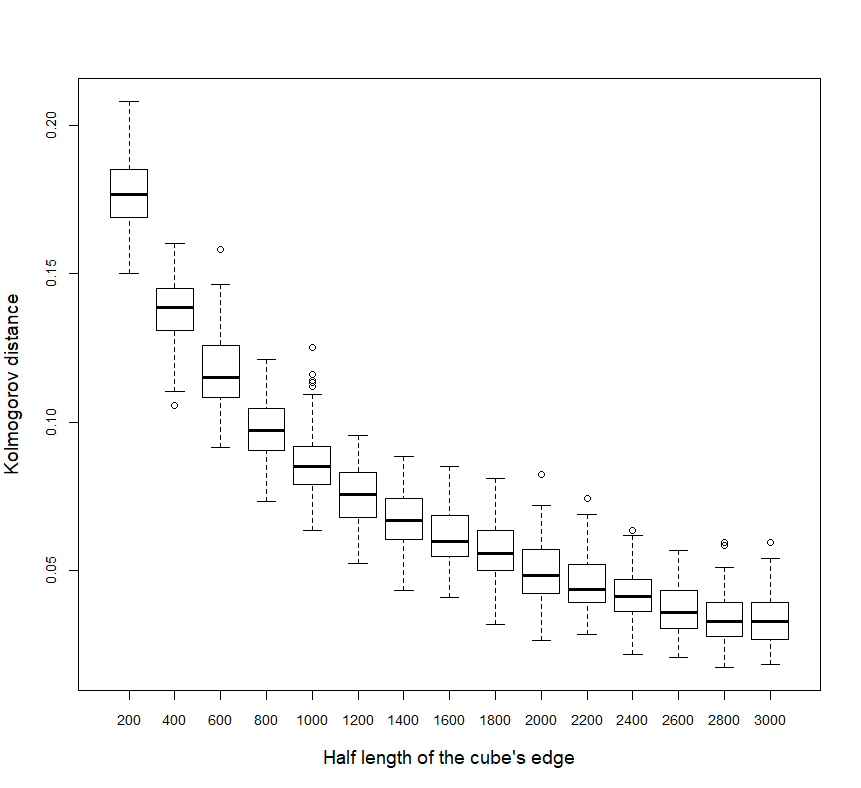}
	\end{minipage} \
	\begin{minipage}{7cm}
		\includegraphics[width=1\linewidth,height=0.8\linewidth]{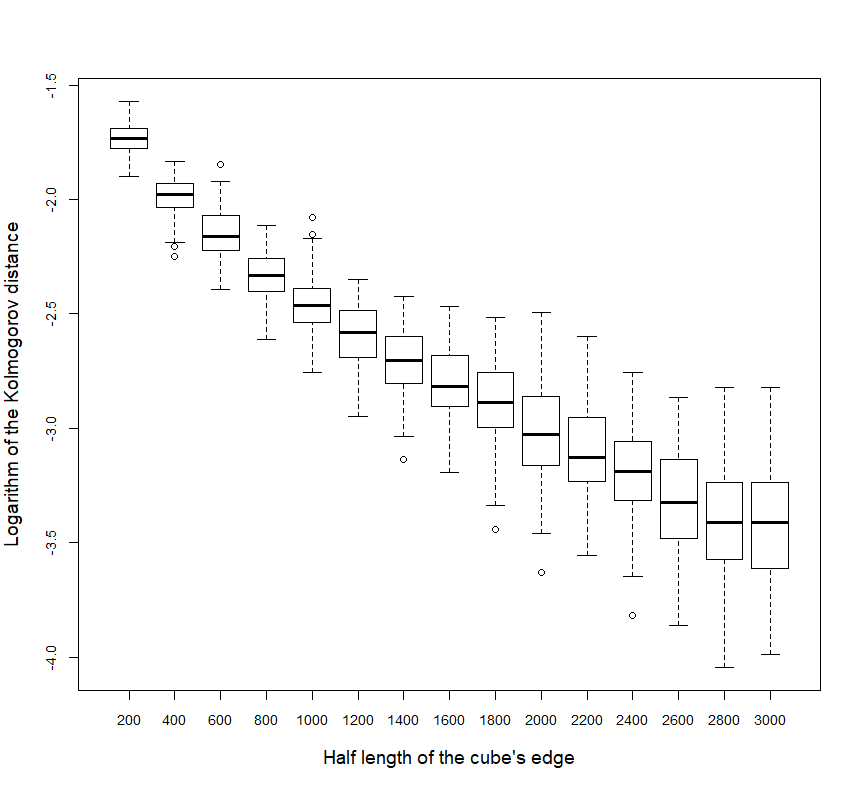}
	\end{minipage}
	\caption{Simulation results for cubes with $R = 4000$ and $h_{sp}(\cdot) = h_{sp}^{c}(\cdot)$.}
	\label{fig6}
\end{figure}

As one can see, in the case of cubes the rate of convergence is also similar in both weighted and unweighted cases. Furthermore, these figures also suggest a power rate of convergence. By fitting the linear regression lines we obtained the equation $\log(\rho\left(X_{r, 2},X_{R, 2}\right)) \approx -1.83 -0.000571r$ in the unweighted case and the equation $\log(\rho\left(X_{r, 2},X_{R, 2}\right)) \approx -1.81 -0.000584r$ in the weighted case. The slopes of these equations are very similar to the slope in (\ref{pwr}), indicating the same rate of convergence in the cases of both surfaces. 

\section{Directions for future research}\label{sec5}

This paper discussed the asymptotic behaviour of least squares estimators in regression models for long-range dependent random fields observed on spheres. The results were obtained under rather general assumptions on the random fields.

We expect that the results derived in Sections~\ref{sec2} and \ref{sec3} also hold in the case of random fields observed on more general surfaces. This conjecture is supported by the results in Section~\ref{sec4}, where asymptotic behaviour of LSE was studied for the cases of sphere and cube surfaces. The obtained results were quite similar in both cases.

The upper bound for the rate of convergence  presented in the paper is the first result of such type for LSE. It would be interesting to obtain sharp upper bounds and extremal examples. The results of Section~\ref{sec4} suggest that the actual speed of convergence of LSE might be a power function of $r$.

The linear regression model was considered in this paper. It would be interesting to obtain analogous results for non-linear regression models.

\section*{Acknowledgements}
This research was partially supported under the Australian Research Council's Discovery Project
DP160101366.\\ This research includes extensive simulation studies using the computational cluster Raijin of the National Computational Infrastructure (NCI), which is supported by the Australian Government and La Trobe University.

\bibliographystyle{tfnlm}
\bibliography{article_03_03}

\end{document}